\documentclass[hidelinks]{article}
\usepackage{graphicx} 
\usepackage[a4paper, margin=2.5cm]{geometry}
\usepackage{hyperref}
\usepackage{amsthm}
\usepackage{bbm}
\usepackage{thm-restate}
\usepackage{geometry}
\usepackage{amssymb}
\usepackage{amsmath}
\usepackage{cleveref}
\usepackage{comment}
\usepackage{xcolor}
\usepackage[shortlabels]{enumitem}
\usepackage{listings}
\usepackage{blkarray}
\usepackage{lmodern}
\usepackage[english]{babel}
\usepackage{float}
\usepackage{tikz}

\newtheorem{theorem}{Theorem}[section]
\newtheorem{lemma}[theorem]{Lemma}

\newtheorem{claim}[theorem]{Claim}

\newtheorem{conjecture}[theorem]{Conjecture}

\newtheorem{problem}[theorem]{Problem}

\theoremstyle{definition}

\newcommand{\eps}{\varepsilon}

\title{Global rigidity of random graphs in $\mathbb{R}$}
\author{
Richard Montgomery\thanks{Institute of Mathematics, University of Warwick, UK. Email: \textbf{richard.montgomery@warwick.ac.uk}. Supported
by the European Research Council (ERC) under the European Union Horizon 2020 research and innovation programme (grant
agreement No. 947978)
} \and
Rajko Nenadov\thanks{School of Computer Science, University of Auckland, New Zealand. Email: \textbf{rajko.nenadov@auckland.ac.nz}. Supported by the New Zealand Marsden Fund.} \and
Julien Portier\thanks{Department of Pure Mathematics and Mathematical Statistics (DPMMS), University of Cambridge, UK. Email: \textbf{jp899@cam.ac.uk}. Supported by the EPSRC (Engineering and Physical Sciences Research Council) and by the Cambridge Commonwealth, European and International Trust.} \and 
Tibor Szabó\thanks{Institute of Mathematics, Freie Universit\"at Berlin, Germany. Email: \textbf{szabo@math.fu-berlin.de}. Research partially funded by the DFG (German Research Foundation) under Germany’s Excellence Strategy – The Berlin Mathematics Research Center MATH+ (EXC-2046/1, project ID: 390685689)}
}
\date{}

\begin{document}

\maketitle

\begin{abstract}
We investigate the problem of reconstructing a set $P\subseteq \mathbb{R}$ of distinct points, where the only information available about $P$ consists of the distances between some of the pairs of points. 
More precisely, we examine which properties of the graph $G$ of known distances, defined on the vertex set $P$, 
ensure that $P$ can be uniquely reconstructed up to isometry.
We prove that as soon as the random graph process has minimum degree 2, with high probability it can reconstruct all distances within any point set in $\mathbb{R}$. This resolves a conjecture of Benjamini and Tzalik.
We also study the feasibility and limitations of reconstructing the distances within almost all points using much sparser random graphs.
In doing so, we resolve a question posed by Girão, Illingworth, Michel, Powierski, and Scott.
\end{abstract}

\section{Introduction}

We consider a set $P\subseteq \mathbb{R}^d$ of (distinct) points and we are interested in the distances between any two of them, but we only have access to the distances between some of them. Which properties of the graph $G$ of known distances on the vertex set $P$ are sufficient for the ``reconstruction'' of $P$, up to isometry? 
One can imagine this problem arising naturally in various scenarios when we are interested in monitoring the relative positions of a set of agents (say a flock of birds in the sky or a fleet of ships in the ocean) moving in space~\cite{Lovasz-book}. The agents are equipped with sensors that allow certain pairs to measure the distance between them, though measuring all pairwise distances is not feasible.
{This leads to several key questions:}
When can we recover the geometric positions of all agents relative to each other? How should we choose $G$ so that we are able to recover the relative positions no matter how the agents are placed? 
What if we are  content already with the relative positions of {\em most} of the agents?

The complete graph on any pointset $P\subseteq \mathbb{R}^d$ of course fully reconstructs $P$, but the aim here would rather be to reconstruct from less than complete information.
To be able to do that when the dimension $d$ is at least $2$, one needs to impose some restriction on $P$, as otherwise there are examples which show that if $G$ is missing just one (carefully chosen) edge, a full reconstruction of $P$ is not possible. For example, consider the configuration with $n-2$ points on a line and two points outside of the line. Then we cannot decide whether these two points lie on the same side of the line or not, unless we are given the distance between them. It turns out that if one restricts to so-called {\em generic} point sets $P$, the coordinates of which are algebraically independent over the rationals, then whether or not $P$ is reconstructible from $G$ depends only on the combinatorial properties of $G$ and not on $P$.  This case has been extensively studied and has a rich mathematical theory (e.g.\ see \cite{barre18rigidrandom,jordan22randomrigid,kasivis11rigid,lew23randomrigid,lovasz82rigid,villanyi2023every}; for a thorough introduction to the topic, see \cite{handbook_discrete_geometry}). 

When $d=1$ however, unlike in the case of higher dimensions, there are no clear obstacles which justify restricting to generic, as opposed to arbitrary, point sets. This property is what we focus on in the present paper.
It  is a folklore result (e.g.\ see \cite[Chapter 63]{handbook_discrete_geometry}) that  the known-distance graph $G$ reconstructs 
any generic point set $P\subseteq \mathbb{R}$  if and only if it is $2$-connected. 
There is a rich family of two-connected graphs however which do not reconstruct certain non-generic point sets in $\mathbb{R}$ (the four-cycle to begin with). 
In fact, for any given $d\in \mathbb{N}$ there are many graphs which do not reconstruct some non-generic point set in $\mathbb{R}$, yet they can reconstruct any generic point set not only in $\mathbb{R}$, but also in $\mathbb{R}^d$. 
Such graphs of arbitrary large connectivity $d(d+1)$ were constructed by Girao et al.~\cite{girao2023reconstructing} and the fact that they reconstruct generic point sets in $\mathbb{R}^d$ follows from the recent breakthrough of Vill\'anyi~\cite{villanyi2023every}. 
Garamv\"olgyi~\cite{garamvolgyi2022global} in fact proved that the family of our interest is pretty complex: it is co-NP-complete to decide whether a given graph reconstructs any set of distinct points in $\mathbb{R}$.

This difficulty makes the typical behaviour in this context particularly interesting. In this direction Benjamini and Tzalik \cite{benjamini2022determining} initiated the study 
of the Erd\H os-R\'enyi random graph $G(n,p)$.  For any point set
$P\subseteq\mathbb{R}$ with $|P|=n$, they showed that for some constant $C$ it holds that if the graph $G$ of known distances is distributed as $G(n, C \log n / n)$, then with high probability (w.h.p.) $P$ is reconstructible from~$G$. 
This result is of course best possible up to the constant factor $C$ as it is well-known that for $p < (\log n  +\log\log n) /n$ the random graph $G(n,p)$ w.h.p.\ is not $2$-connected, hence does not even reconstruct generic point sets.   
Another important aspect of the above result is that we are {\em first} given $P$ in $\mathbb{R}$ and {\em only after that} we randomly generate the graph of known distances, which, with high probability, is proven to reconstruct all distances within that specific set $P$.

In both of these directions Benjamini and Tzalik \cite{benjamini2022determining} conjectured a significant strengthening of their result: namely, that (i) a hitting time result should also hold, i.e.\ that the reconstruction property for a particular $P$ should likely be true from the very moment the random graph process becomes $2$-connected, and (ii) at that point, not only does the random graph reconstructs the given point set, but it reconstructs any such set.
Recently Gir\~{a}o, Illingworth, Michel, Powierski, and Scott~\cite{girao2023reconstructing} proved (i). In our first main result, stated below, we prove (ii) using a different and simpler approach. We also study the possibilities/limitations of the reconstruction of a linear sized portion of the point set.

\subsection{New results.}
\paragraph{Terminology.} Consider a random graph process $\{G_m\}_{m \ge 0}$ on the vertex set $[n]$, for some $n \in \mathbb{N}$, where $G_0$ is an empty graph and each $G_i$ is formed from $G_{i-1}$ by adding a new edge uniformly at random. Let $\tau := \tau_2$ denote the smallest $m$ such that $\delta(G_m) \ge 2$. \\
A {\em (bar-and-joint) framework (in $\mathbb{R}^d$)} is a pair $(G,f)$ where $G$ is a simple graph and $f:V(G) \rightarrow \mathbb{R}^d$ maps the vertices into the $d$-dimensional euclidean space. 
For a graph  $G=(V,E)$ and an injective function $f:V  \rightarrow \mathbb{R}^d$, we call the framework  $(G,f)$ {\em globally rigid} if for every injective function $g:V\rightarrow \mathbb{R}^d$ 
with the property that $||f(x)-f(y)|| = ||g(x)-g(y)||$ for every edge $xy\in E$, we also have $||f(x)-f(y)|| = ||g(x)-g(y)||$ for every pair $x,y \in V$. To express this in our informal description above, we said that ``$G$ reconstructs the point set $f(V)$ (up to isometry)". 
We call a graph $G$ {\em globally rigid (in $\mathbb{R}^d$)} if for every injective function $f:V\rightarrow \mathbb{R}^d$ the framework $(G,f)$ is globally rigid.

\smallskip

\noindent{\bf Remark.} We note that the term ``globally rigid'' in the literature is often used in the sense that {\em generic} embeddings are reconstructible. In contrast, as introduced above,  ``globally rigid'' in our paper will always mean that {\em injective} embeddings are reconstructible. Garamv\"olgyi~\cite{garamvolgyi2022global} refers to this as ``injectively globally rigid''. As this is the only type of rigidity we deal with, we omit the adjective. 

\smallskip


Using this terminology, the theorem of Girão, Illingworth, Michel, Powierski, and Scott~\cite{girao2023reconstructing}, mentioned earlier, can be restated as follows.

\begin{theorem}[\cite{girao2023reconstructing}] \label{thm:oxford} 
    For every injective function  $f \colon [n] \rightarrow \mathbb{R}$ the framework $(G_{\tau}, f)$ is globally rigid w.h.p. 
\end{theorem}

The proof of this result follows the approach laid down by \cite{benjamini2022determining}. In our first main result we prove that $G_{\tau}$ with high probability not only reconstructs the  given embedding $f$ but reconstructs any embedding. This confirms the conjecture of Benjamini and Tzalik. Our proof strategy is different (and simpler) from the previous ones.

\begin{theorem} \label{thm:main}
 $G_{\tau}$ is globally rigid in $\mathbb{R}$ w.h.p.  
 \end{theorem}

As the hitting time of $2$-connectivity is equal to $\tau$ w.h.p., our result shows that on the level of  ``typical" graphs there is no difference  in $\mathbb{R}$  between generic global rigidity and our less restrictive injective version. 
For the same reason, for random graphs sparser than $G_{\tau}$ one cannot hope in general for the reconstruction of the full point set. Were we content ourselves with reconstructing only almost every point, then we can do much better.  The bottleneck here again turns out to be $2$-connectivity. It is well-known that  a $2$-connected component of size $(1-o(1))n$ emerges in $G(n,p)$ when $p=\omega( 1/n)$.  
Gir\~{a}o et al.\ \cite{girao2023reconstructing} showed that  in the same regime it also holds that given any injective $f \colon [n] \rightarrow \mathbb{R}$, the random graph $G(n,p)$ w.h.p. reconstructs $f$ on some  subset $V' \subseteq [n]$ of size $|V'| = (1 - o(1))n$. 

They asked whether $1/n$ is also the threshold for the property that $G(n,p)$ reconstructs a constant fraction of points for {\em any} injective function $f$. 
Using our novel approach developed for Theorem~\ref{thm:main} we can show that this holds in a much stronger form. Namely, one can always reconstruct the same subset of vertices, independent of $f$, the size of which is a fraction of $n$ arbitrarily close to $1$. 

\begin{theorem} \label{thm:sparse}
For every  $\varepsilon >0$ there exists $C \in \mathbb{R}$ such that in the random graph $G\sim G(n,C/n)$ w.h.p. there exists a subset $V' \subseteq [n]$ of size $|V'| \geq (1-\varepsilon)n$ for which the induced subgraph $G[V']$ is globally rigid in $\mathbb{R}$. 
\end{theorem}

Moreover, Gir\~{a}o et al.\ \cite{girao2023reconstructing} conjectured that for a given injective embedding it holds w.h.p. right after the point $p=1/n$ of the phase transition for the emergence of the giant $2$-connected component in $G(n,p)$, that a positive fraction of the points can also be reconstructed. 

\begin{conjecture}
\label{conj:Girao}Let $f : [n] \rightarrow \mathbb{R}$ be an arbitrary injective function and $\varepsilon >0$. Then for every $p \ge (1+\eps)/n$, in the random graph $G \sim G(n,p)$ w.h.p. there exists a subset $V' \subseteq [n]$ of size $|V'| = \Omega_{\varepsilon}(n)$ such that the restricted framework $(G[V'], f|_{V'})$ is globally rigid in $\mathbb{R}$.
\end{conjecture}

We do not settle this conjecture, but we show that its strengthening to global rigidity  does not hold.

\begin{theorem}
\label{thm:LowerBoundGlobalRigidity}   There exists $\gamma >0$ such that for any $p < 1.1/n$, the random graph $G \sim G(n,p)$ w.h.p. satisfies that for every subset $V' \subseteq [n]$ of size $|V'| \geq \gamma \log n$, the induced subgraph $G[V']$ is not globally rigid in $\mathbb{R}$. 
\end{theorem}

We prove Theorem~\ref{thm:main} and Theorem~\ref{thm:sparse} in Section~\ref{sec:proof}. In Section~\ref{sec:not_rigid} we prove Theorem~\ref{thm:LowerBoundGlobalRigidity}, before finishing with some open problems in Section~\ref{sec:open}.

\paragraph{Acknowledgement.} The first, second and fourth authors thank the research institute MATRIX, in Creswick, Australia, where part of this research was performed, for its hospitality, and the organisers and participants of the workshop on {\em Extremal Problems in Graphs, Designs, and Geometries} for a stimulating research environment. The third author is extremely grateful to Julian Sahasrabudhe for suggesting the problem and for his support. The authors also thank Douglas Barnes, Thomas Lesgourgues, Brendan McKay, Jan Petr, Benedict Randall Shaw, Marcelo De Sa Oliveira Sales and Alan Sergeev for productive discussions.

\section{Criteria for global rigidity}\label{sec:proof}

The following lemma is the crux of our proofs of Theorem \ref{thm:main} and Theorem \ref{thm:sparse}, which are then derived as simple corollaries from it.

\begin{lemma} \label{lemma:main}
    Let $G$ be a graph with $V(G)=[n]$, and suppose it satisfies the following two properties:
    \begin{enumerate}[(P1)]
        \item \label{prop:edge} For every disjoint $U, W \subseteq V(G)$ of size $|U|,|W| \ge n/15$ there is an edge in $G$ between $U$ and $W$.
        \item \label{prop:2nbr} For every $U \subseteq V(G)$ of size $n/15 \le |U| < n$, there exists a vertex $v \in V(G) \setminus U$ with at least two neighbors in $U$.
    \end{enumerate}
Then $G$ is globally rigid in $\mathbb{R}$.
\end{lemma}  


\begin{proof} 
The result holds trivially for $n = 1$, thus we assume $n > 1$. Let $f$ and $g:[n] \rightarrow \mathbb{R}$ be two injective functions such that $|f(x)-f(y)| = |g(x)-g(y)|$ for every edge $xy\in E(G)$. We show that the same equality holds for any two vertices $x,y\in V(G)$, i.e. $f$ and $g$ are isometric. Equivalently, this means that there exists $a\in \{ 1, -1\}$ and $b \in \mathbb{R}$ such that $f = a f' + b$. Let
\begin{align*}
    L_f &:= \bigl\{ i \in [n] : |\{ x \in [n] : f(i) < f(x) \}| \geq \lceil n/2 \rceil  \bigr\} \\
    R_f &:= \bigl\{ i \in [n] : |\{ x \in [n] : f(x) < f(i) \}| \geq \lceil n/2 \rceil \bigr\}
\end{align*}
be the {\em left-half} and {\em right-half} of $f$ (omitting the middle vertex when $n$ is odd), and define $L_g$ and $R_g$ analogously. 

We can assume without loss of generality that $|L_f\cap L_g|\ge \lceil (|L_f|-1)/2\rceil \ge \lceil (n - 3) / 4 \rceil$. Otherwise we consider instead the function $-g$, which is isometric to $g$, and use that $L_{-g}=R_g$. As $n > 1$, by \ref{prop:2nbr} we necessarily have $n>15$ (otherwise $G$ does not satisfy \ref{prop:2nbr}), thus $|L_f\cap L_g| > n/5$. Then $|R_f\cap R_g| = |R_f| + |R_g| - |R_f\cup R_g|\geq n-1 - |\overline{(L_f\cap L_g)}| \geq n/5$. Set $L := L_f \cap L_g$ and $R := R_f \cap R_g$.

We first prove that the induced bipartite graph~$G[L, R]$ contains a connected component spanning at least $n/15$ vertices. Let $C^{(1)},\ldots,C^{(k)}$ be any ordering of the connected components of~$G[L, R]$. Toward a contradiction, assume that each connected component contains fewer than $n/15$ vertices, i.e.\ for all $j\in[k]$, we have $|C^{(j)}| < n/15$. Let $i>1$ be the smallest index such that 
$$
    \sum_{j=1}^i|C^{(j)}\cap L| \ge n/15 \quad \text{or} \quad \sum_{j=1}^i|C^{(j)}\cap R| \ge n/15.
$$
Without loss of generality assume that $\sum_{j=1}^i|C^{(j)}\cap R|\leq \sum_{j=1}^i|C^{(j)}\cap L|$. Then, using $|C^{(j)}| < n/15$ for all $j\in [k]$, by the minimality of $i$ we have
$$
    \sum_{j=1}^i|C^{(j)}\cap R|\leq \sum_{j=1}^i|C^{(j)}\cap L| < 2n/15.
$$
As $\sum_{i = 1}^k |C^{(j)} \cap R| = |R| \ge n / 5$, we then have
$$
    \sum_{j=i+1}^k|C^{(j)}\cap R|\geq n/15.
$$
Then by \ref{prop:edge} there exists an edge between $\cup_{j=1}^iC^{(j)}\cap L$ and $\cup_{j=i+1}^kC^{(j)}\cap R$, which contradicts the assumption that $C^{(1)}, \ldots, C^{(k)}$ are the connected components of $G[L, R]$. 

Let $C$ be the vertices of the largest connected component of~$G[L, R]$. As we have just showed, $|C|\geq n/15$. Let $y_1 \in C \cap L$ be an arbitrary vertex and let $g' = g - g(y_1) +f(y_1)$, i.e.\ the translation of $g$ that agrees with $f$ on $y_1$. As this is not changing the relative order with respect to $g$, we have $L_g=L_{g'}$ and $R_g=R_{g'}$. Let $U:= \{ u \in [n]: f(u) = g'(u) \}$ be the set of vertices on which $f$ agrees with $g'$. By the definition, we have $y_1 \in U$. We claim that the whole connected component $C$ is contained in $U$. This is because for any vertex $x\in U$ and edge $xy\in E(G)$ of $C$, we also have $y\in U$: Suppose first that $x\in L_f\cap L_{g'}$ and $y\in R_f\cap R_{g'}$; the other case is analogous. 
Since $x$ is in the left-half of $f$, $y$ is in the right-half of $f$, we have $f(y) = f(x) + |f(x)-f(y)|$. Similarly, we obtain $g'(y) = g'(x) + |g'(x)-g'(y)|$. Using that  $|f(x)-f(y)|=|g'(x)-g'(y)|$ and $x\in U$, we conclude that $f(y) = g'(y)$, so $y\in U$.

Now $C \subseteq U$ implies $|U| \geq n/15$. If $U=[n]$ we are done. Otherwise \ref{prop:2nbr} can be applied and we take a vertex $v\in V(G) \setminus U$ which has two neighbors $u_1, u_2$ in $U$. Assume, by relabelling if necessary, that $f(u_1)=g'(u_1)<f(u_2)=g'(u_2)$.
The $f$-value of $v$ is determined by $f(u_1), f(u_2), |f(u_1)-f(v)|$ and $|f(u_2)-f(v)|$. Indeed,
depending on whether 
$f(u_2)$, $f(u_1)$ or $f(v)$ is in between the other two, $f(v)$ is equal to $f(u_1) + |f(u_1)-f(v)| = f(u_2)+ |f(u_2)-f(v)|$, $f(u_1) -  |f(u_1)-f(v)| = f(u_2)- |f(u_2)-f(v)|$, or $f(u_1) + |f(u_1)-f(v)| = f(u_2)- |f(u_1)-f(v)|$, respectively. The $g'$-value of $v$ is determined by $g'(u_1), g'(u_2), |g'(u_1)-g'(v)|$ and $|g'(u_2)-g'(v)|$ analogously. Finally, since $f(u_1)=g'(u_1)$ and $f(u_2)=g'(u_2)$, the values $f(v)$ and $g'(v)$ coming from the analogous formulas also agree. 
That means $v\in U$, contradicting $v\in V(G) \setminus U$. Thus, $U=V(G)$, and therefore, as $f=g'$, $f$ and $g$ are isometric.
\end{proof}

\subsection{Applications}

We need the following simple property of random graphs.

\begin{lemma} \label{lemma:random}
    For every $\eps > 0$ there exists $C > 0$ such that if $m \ge C n$, then $G \sim G(n,m)$ w.h.p.\ has the following property: 
    \begin{enumerate}[(P1)]
        \setcounter{enumi}{2}
        \item \label{p3} For every disjoint $X, Y \subseteq V(G)$ of size $|X|, |Y| \ge \eps n$, there exists an edge between $X$ and $Y$ in $G$.
    \end{enumerate}
\end{lemma}
\begin{proof}
    For fixed $X$ and $Y$, the probability that there is no edge between $X$ and $Y$ is
    $$
        \binom{\binom{n}{2} - |X||Y|}{m} / \binom{\binom{n}{2}}{m} \le e^{-|X||Y| m  / n^2} \le e^{-\eps^2 C n}.
    $$
    There are at most $2^{2n}$ ways to choose $X$ and $Y$, thus, for $C > 2 / \eps^2$, w.h.p.\ this bad event does not happen for any such pair of sets.  
\end{proof}

With Lemma \ref{lemma:main} and Lemma \ref{lemma:random} at hand, the proofs of Theorems \ref{thm:main} and \ref{thm:sparse} are straightforward.

\begin{proof}[Proof of Theorem \ref{thm:main}]
For the proof of Theorem~\ref{thm:main} we check that w.h.p.\ both \ref{prop:edge} and \ref{prop:2nbr} hold for $G_{\tau}$, so that the result follows directly by Lemma~\ref{lemma:main}. 

Let $C$ be a constant given by Lemma \ref{lemma:random} for $\eps = 1/31$. It is well known~\cite{bollobasbook} that w.h.p.\ $\tau \ge C n := m$ (with $C$ as given by Lemma~\ref{lemma:random}, or, indeed, any constant $C$). As $G_m$ is uniformly distributed among all graphs with $n$ vertices and exactly $m$ edges, by Lemma \ref{lemma:random} we have that w.h.p.\ \ref{p3} holds in $G_m$. Since \ref{p3} is monotone, it also holds in $G_\tau$.

Property \ref{p3} is straightforwardly stronger than \ref{prop:edge} and also implies \ref{prop:2nbr} in the case $n/15 \le |U| \le n/2$. Indeed, for the latter let $S \subseteq V(G) \setminus U$ be a subset of size $\eps n$. By \ref{p3} we have
$$
    |N(S) \cap U| \ge |U| - \eps n > |S|,
$$
thus there exists a vertex in $S$ with two neighbors in $U$. The remaining case $|U| > n/2$ of the property \ref{prop:2nbr} is proven to hold w.h.p., for example, in \cite[Proposition 2.3]{lew23randomrigid}.
\end{proof}

\begin{proof}[Proof of Theorem \ref{thm:sparse}]
For convenience we prove our result for the $G(n,m)$ random graph model  where a graph is chosen uniformly at random among all labeled graphs with $n$ vertices and $m$ edges. Namely, for every  $\varepsilon >0$ we show that there exists $C \in \mathbb{R}$ such that for $G\sim G(n,Cn)$ w.h.p. both \ref{prop:edge} and \ref{prop:2nbr} hold. The equivalent statement for $G(n,p)$ follows by \cite[Theorem 1.4]{frieze16book}.

We can assume $\eps > 0$ is sufficiently small, take $C$ from Lemma~\ref{lemma:random} applied to $\eps$ and let $m\geq Cn$. By Lemma \ref{lemma:random},  $G \sim G(n,m)$ w.h.p.\ has the property \ref{p3}. This immediately implies \ref{prop:edge}, thus to apply Lemma \ref{lemma:main} we just need to find a large subset $V' \subseteq V(G)$ such that $G'=G[V']$ satisfies \ref{prop:2nbr}. 
        We define $V':= V(G) \setminus A$, where $A \subseteq V(G)$ is a largest subset such that $|A| \le \eps n$ and $|N(A)| \le |A|$.

    To check \ref{prop:2nbr} for a subset $U\subseteq V'$ of size $|V'| / 15 \le |U| < |V'| - \eps n$ we consider a subset $S \subseteq V' \setminus U$ of size $\eps n \leq |V'\setminus U|$. Applying \ref{p3} we have
    $$
        |N(S) \cap U| \ge |U| - \eps n > |S|,
    $$
    which verifies that there is a vertex in $S$ with two $G'$-neighbors in $U$, for otherwise $|N(S)\cap U|\leq |S|$.

If $U \subseteq V'$ is of size $|U| \ge |V'| - \eps n$, then we claim that for $S=V'\setminus U$ we have $|N(S) \setminus A| > |S|$, which in turn implies that some vertex of $S$ has two neighbors in $U$. 
Otherwise we have 
    $$
        |N(A \cup S)| \le |A| + |S| = |A\cup S|,
    $$
    which implies $\eps n < |A \cup S| \le 2 \eps n$ by the maximality of $A$. We can then apply \ref{p3} to obtain
    $$
        |N(A \cup S)| \ge n - |A \cup S| - \eps n > |A \cup S|,
    $$
    a contradiction. 
\end{proof}

\section{Sparse random graphs are typically far from globally rigid}
\label{sec:not_rigid}

The next lemma gives a general condition under which a graph is not globally rigid in $\mathbb{R}$.  

\begin{lemma}
\label{lem:CutNotRecontructible}
    Suppose the vertex set of a graph $G\neq K_2$ can be partitioned into non-empty sets $A$ and $B$ such that  the maximum degree of $G[A,B]$ is at most $1$.
    Then $G$ is not globally rigid.
\end{lemma}
\begin{proof}
We define two non-isometric injective functions $f$ and $g: V(G) \rightarrow \mathbb{R}$ such that $|f(x)-f(y)|=|g(x)-g(y)|$ for every edge $xy\in E(G)$. 
   For $f$ we embed $A$ injectively in $[0,1]$ arbitrarily, and, for every $b \in B$ such that there exists $a \in A$ with $ab \in E(G)$, we embed $b$ at $f(a) +10$. 
    We embed the rest of $B$ injectively and arbitrarily inside $[10,11] \setminus (f(A)+10)$. 
    We now define the injective $g$ such that $g\arrowvert_{A}=f\arrowvert_{A}$ and for every $b \in B$, we set $g(b)=f(b)-20$. Note that, for an edge $xy\in E(G)$ inside $A$ or inside $B$ the property 
    $|f(x)-f(y)|=|g(x)-g(y)|$ is immediate. For and edge $ab\in E(G)$ with $b\in B$ and $a\in A$, we have that $g(b)=f(a)-10$, so $|g(a)-g(b)|=10=|f(a)-f(b)|$. Furthermore, for every $a \in A$ and $b \in B \setminus N(a)$, we have 
    \begin{align*}
        |g(a)-g(b)|-|f(a)-f(b)|=(f(a)-(f(b)-20))-(f(b)-f(a))=2(10-|f(a)-f(b)|)\neq 0,
    \end{align*}
     since $f(b)\not\in f(A)+10$. This shows that $f$ and $g$ are not isometric, since $G[A,B]$ is not a complete bipartite graph. \end{proof}

The next lemma gives us a structural property of sparse random graphs from which we deduce Theorem \ref{thm:LowerBoundGlobalRigidity}.
 The proof is based on a modified version of the classical study of the sizes of the connected components of sparse random graphs using comparisons to Galton-Watson processes.
We first recall the following Chernoff bounds (see, for instance, Theorem 2.1 in \cite{JLR}) as well as Azuma's inequality (see, for instance, \cite{wormald1999differential}, for the result and the definition of a submartingale).
\begin{lemma} \label{lem:chernoff}
Let $X$ be a binomial random variable with expected value $\mu$. Then, for any $\delta\in (0,1)$, we have 
$\mathrm{Pr}(X > (1+\delta)\mu) \leq e^{\frac{-\delta^2\mu}{2+\delta}}.$
\end{lemma}

\begin{lemma}\label{lem:azuma}
Let $(X_i)_{i \geq 0}$ be a submartingale and let $c_i>0$ for each $i\geq 1$. If $|X_i -X_{i-1}| \leq c_i$ for each $i\geq 1$, then, for each $n\geq 1$,
\[
\mathrm{Pr}(X_n\leq X_0- \eps ) \leq  \exp \left( -\frac{\eps^2}{2\sum_{i=1}^n c_i^2} \right).
\]
\end{lemma}

\begin{lemma}
\label{lem:GW2}
    Let $G \sim G(n,p)$ for $p \leq \frac{1.1}{n}$. Then, w.h.p., for each $v\in V(G)$ there is some $A_v\subset V(G)$ with $v\in A_v$ and $|A_v|\leq 6000\log n$ such that the maximum degree of $G[A_v,V(G)\setminus A_v]$ is at most $1$.
\end{lemma}
\begin{proof}
Note that we can suppose for the rest of the proof that $G \sim G(n,p)$ for $p=1.1 / n$, as the property we are interested in is clearly non-decreasing in $p$. Let $v\in V(G)$. We will show that the property in the lemma holds for $v$ with probability $1-o(n^{-1})$, so that the result follows by a simple union bound.

It may be useful to recall the following classical process to study the component of the sparse random graph $G$ containing $v$. 
We begin by setting $X=\{v\}$, $X^-=\emptyset$ and $Z=V(G)\setminus \{v\}$. At each step, if possible, we choose an arbitrary vertex $w$ from $X\setminus X^-$, move its neighbors in $Z$ from $Z$ into $X$ and add $w$ to $X^-$. 
{It is well known, via comparison to an appropriate Galton-Watson process, that when the edge probability is $c/n$ for any constant $c>1$, then $X$ becomes linear-sized (in $n$) with some probability bounded away from 0.}
{For our purpose, we modify this process slightly to ensure that, even with an edge probability of $p=1.1/n$, the set $X$ remains small with probability $1-o(n^{-1})$. Our key change is the introduction of an (initially empty) set $Y$, such that if $w$ has exactly one neighbor not in $X$, which is not in $Y$, we instead move that neighbor from $Z$ to $Y$.}
This adjustment uses the fact that the set we eventually find, $A_v$, for the application of Lemma~\ref{lem:CutNotRecontructible}, can have a non-empty neighborhood in $G$, so we only definitely wish to add a vertex to $X$ if it has more than 1 neighbor in $X^-$ or if it has a neighbor in $X^-$ which has more than 1 neighbor outside of $X$. This is enough to drop the expected number of vertices we add to $X$ at each stage below 1, so that the process will quickly terminate. We now carefully give this process in full.

We define $A_v$ via a process, starting with $A_v=\{v\}$. 
In each step we add a (well-chosen) new vertex to $A_v$ and reveal the edges of $G$ incident to it. 
To aid us choosing the next vertex, we maintain sets $X, Y, Z, X^-$,    
such that throughout the whole process $X\cup Y \cup Z = V(G)$ is a partition of the vertex set, $X^- \subseteq X$, $\Delta (G[X^-,Y\cup Z]) \leq 1$, and moreover
 \begin{equation*} \label{eq:process}
|N(u) \cap X^-| = 
\begin{cases} 1 & \text{if $u\in Y$} \\ 0 & \text{if $u\in Z$.}
\end{cases}
\end{equation*}
In the set $X$ we will be collecting the vertices which eventually end up in $A_v$, and its subset $X^-$ contains those vertices whose neighbors were already revealed by our process.

We index the sets $X,X^-, Y, Z$ by the number $i$ of the current step and initialize by defining  $X_0=\{v\}$, $X_0^-=Y_0=\emptyset$ and $Z_0=V(G)\setminus \{v\}$. Let $\sigma = 6000 \log n$ be the number of rounds we target our process to likely end in. For $i \ge 0$, do the following: 
\begin{enumerate}
\item If $i \le \sigma$ and $|Z_i| {<} \frac{99 n}{100}$ then terminate with failure.
\item If $X_i = X_{i}^-$ then set $A_v = X_i$ and terminate with success.
\item Pick an arbitrary $v_i\in X_{i}\setminus X_{i}^-$ and move it to $X^-$, that is, we define $X_{i+1}^- =X_{i}^-\cup \{v_i\}$.
\item Reveal the set $W_i = N(v_i) \cap (Y_{i}\cup Z_{i})$ of neighbors of $v_i$ outside of $X_{i}$.
\begin{itemize}
\item If $W_i\cap Y_{i} = \emptyset$ and $|W_i \cap Z_{i}| \leq 1$ then we update our sets by $X_{i+1}=X_{i}$, $Y_{i+1}=Y_{i}\cup W_i$, and $Z_{i+1}=Z_{i}\setminus W_i$.
\item Otherwise, we update by $X_{i+1}=X_{i}\cup W_i$, $Y_i=Y_{i}\setminus W_i$, and $Z_{i+1}=Z_{i}\setminus W_i$.
\end{itemize}
\end{enumerate}
It is straightforward to check by induction that for these updates the conditions we promised are indeed maintained. In each round $X^-$ grows exactly by one element, hence for every $i\geq 0$ we have $i = |X^-_i| \leq |X_i| \le n$.
Therefore the number $t$ of rounds our procedure takes until termination is always at most $n$.

If the procedure terminates with success in at most $\sigma$ rounds then $A_v=X_t=X_t^-$ and  $|A_v| = |X_t^-|=t \le \sigma$. Furthermore, since $V(G) = X_t \cup  Y_t\cup Z_t$, we also have  $\Delta(G[A_v, V(G)\setminus A_v]) = \Delta(G[X_t^-, Y_t\cup Z_t]) \leq 1$, as desired. We will show that this happens with probability $1 - o(n^{-1})$ and use the union bound for $v\in V$ to conclude the proof of the lemma. 


If the procedure terminates with failure then $t=i$ for some $i\leq \sigma$ and $|Z_i|<\frac{99 n}{100}$. Then among the at most $in\leq \sigma n$ pairs we revealed in $G$ we will have found at least $\frac{n}{100}$ edges. The probability of this is at most
$$ \Pr\left(\mathrm{Bin}(\sigma n, p) \ge \frac{n}{100}\right) \leq {\sigma n\choose n/100} p^{n/100}= o(n^{-2}).$$ 
Thus the probability of terminating with failure  is $o(n^{-1})$.


It remains to show that the probability of terminating after more than $\sigma$ rounds is also $o(n^{-1})$. To that end for each $0\leq i<t$ we let $n_i=|X_{i+1}\setminus X_{i}|$, and for each $i\geq t$ let $n_i=0$.
Then $|X_0| + \sum_{i=0}^{t-1} n_i =|X_t|=t$. 
For each $i\geq 0$ and $j=1, \ldots, n$,  let $b_i^j$ be independent Bernoulli variables with success probability $p$. For each $i\geq 0$ we use $b_i^1, \ldots , b_i^{|Y_i|+|Z_i|}$ to generate $W_i$ and hence $Y_{i+1}, Z_{i+1}, X_{i+1}$ and $n_i$.
Furthermore, for each $i\geq 0$ we let $m_i = \sum_{j=1}^{n} b_i^j$, unless $i\geq t$ and we terminated with failure, in which case we let $m_i=1$.  \\
Clearly, for every $i\geq 0$ we always have $m_i\geq n_i$.
We now prove that the difference is non-zero with at least $1/3$ probability.
\begin{claim}
\label{cl:BoundProba_mi_ni}
    Let $i\geq 0$. Conditioned on all $b^j_l$ with $0\leq l\leq i-1$ and $1\leq j \leq n$,  the probability that $m_i-n_i \geq 1$ is at least $1/3$.
\end{claim}

\begin{proof}
If $t \leq i$ then, depending on whether we terminate  with failure or success, the probability that $m_i-n_i=m_i\geq 1$ is 1 or $1-(1-p)^n\geq 1-e^{-1.1}>1/3$, respectively. \\
If  $t > i$ then the probability that $m_i-n_i \geq 1$ is at least the probability that $b_i^j=1$ for no $i$ with $1\leq i\leq |Y_i|$ and exactly one $i$ with $|Y_i|<i\leq |Y_i|+|Z_i|$, which happens with probability $|Z_i|p(1-p)^{|Y_i|+|Z_i|-1}$. As we did not terminate with failure, we have $|Z_i| \geq \frac{99}{100}n$. So 
the probability in question is at least 
$\frac{99n}{100} \frac{1.1}{n}(1-p)^{n} \geq 1.089\cdot e^{-1.1 \times 1.01} > 1/3$, where the next to last inequality holds for large enough $n$. 
\end{proof}

From \Cref{cl:BoundProba_mi_ni} we deduce that the sequence of variables 
$$
    S_i = \sum_{j = 0}^i \min\{m_i - n_i - 1/3, 1\}
$$
forms a submartingale with $|S_i - S_{i-1}| \leq 1$. By Azuma's inequality (Lemma~\ref{lem:azuma}) applied with $\eps = \sigma / 12$, we conclude that with probability $1-o(n^{-1})$
\begin{align}
\label{eq:BoundDiffmim'i}
    \sum_{i=0}^{\sigma-1}(m_i-n_i) \geq \frac{1}{4}\sigma.
\end{align}
Applying Lemma~\ref{lem:chernoff} 
 with $\mu = 1.1 \sigma$ and $\delta= \frac{1}{44}$,  with probability $1-o(n^{-1})$ we have that
\begin{align}
\label{eq:Boundmi}
    \sum_{i=0}^{\sigma -1} \sum_{j=1}^n b_i^j \leq \frac{9}{8}\sigma.
\end{align}
We claim that if $t > \sigma$ then \eqref{eq:BoundDiffmim'i} or  \eqref{eq:Boundmi} does not hold, and hence the probability of this event is $o(n^{-1})$. 
Indeed, otherwise, by subtracting \eqref{eq:BoundDiffmim'i} from \eqref{eq:Boundmi} and using that $ \sum_{i=0}^{\sigma -1} m_i = \sum_{i=0}^{\sigma -1} \sum_{j=1}^n b_i^j$ 
(as we do not terminate with failure)  we obtain that $|X_{\sigma}|= 1+\sum_{i=0}^{\sigma -1} n_i \leq 1 + \frac{7}{8}\sigma$. 
However, this contradicts $\sigma \leq |X_\sigma| $, which follows since $\sigma < t$. 

We have thus shown that the probability of termination with success in at most $\sigma$ rounds is $1-o(n^{-1})$, as desired.
\end{proof}

We can now deduce \Cref{thm:LowerBoundGlobalRigidity}, as follows.
\begin{proof}[Proof of \Cref{thm:LowerBoundGlobalRigidity}]
    Let $G \sim G(n,p)$ for $p \leq \frac{1.1}{n}$. 
    By \Cref{lem:GW2}, w.h.p., for each $v\in V(G)$ there is some $A_v\subset V(G)$ with $v\in A_v$ and $|A_v|\leq 6000\log n$ such that $\Delta(G[A_v,V(G)\setminus A_v]) \leq 1$.
    Let $V' \subseteq V(G)$ be a subset of size $|V'| > 10^4 \log n$. Then, letting $v\in V'$ be arbitrary, we have that $\Delta( G[A_v\cap V',V'\setminus A_v]) \leq 1$, and, as $v\in A_v$, $|A_v|\leq 6000\log n$ and $|V'|>10^4\log n$, that $A_v\cap V'$ is non-empty and that $|V'\setminus A_v| \geq 2$. Therefore, by \Cref{lem:CutNotRecontructible}, $G[V']$ is not globally rigid.
    \end{proof}

\section{Open problems}\label{sec:open}

\paragraph{Random regular graphs.} Following Benjamini and Tzalik \cite{benjamini2022determining}, we studied the problem of global rigidity of random graphs in $\mathbb{R}$. A related natural question is, for which $d$ is a random $d$-regular graph globally rigid in $\mathbb{R}$ with high probability? The methods developed in this article, in particular Lemma~\ref{lemma:main}, together with known estimates on the second largest absolute eigenvalue of random $d$-regular graphs (e.g. see Friedman~\cite{friedman2003proof}) and the Expander Mixing Lemma, imply that a random $d$-regular graph is w.h.p.\ globally rigid in $\mathbb{R}$ for every $n \ge d \ge d_0$, where $d_0$ is a sufficiently large constant. The following argument shows $d_0 \ge 4$.

\begin{theorem}
\label{thm:RegGraphs}
    A $3$-regular random graph is w.h.p.\ not globally rigid in $\mathbb{R}$.
\end{theorem}
\begin{proof}
    Let $G$ be a random $3$-regular graph, and let $C$ be a shortest cycle of $G$. Let $F_1$ be the graph which consists of $2$ triangles sharing one edge exactly, and $F_2$ be the graph which consists of a cycle $u_1u_2u_3u_4$ and a vertex $x$ along with the edges $xu_1$ and $xu_3$. 
    It is known that w.h.p.\ $G$ does not contain any subgraph isomorphic to $F_1$ nor $F_2$ (see for instance Lemma 2.7 in the survey by Wormald \cite{wormald1999models}). If a vertex $b \in V(G) \setminus C$ would have two neighbors in $C$, then we would either obtain a shorter cycle or a copy of $F_1$ or $F_2$ -- neither of which can happen. Therefore, applying \Cref{lem:CutNotRecontructible} to $A=C$ and $B=V(G) \setminus C$ finishes the proof.
\end{proof}
\noindent {\bf Remark.} Pawe\l{} Rza\.{z}ewski~\cite{Rzazewski} pointed out that an elaboration on the above argument shows that {\em every} $3$-regular graph on at least eight vertices is not globally rigid. 

\Cref{thm:RegGraphs} leaves open the problem of determining the smallest $d \ge 4$ for which a random $d$-regular graph with $n$ vertices is globally rigid in $\mathbb{R}$ w.h.p., where we make the following conjecture.

\begin{conjecture}
    A random $4$-regular graph with $n$ vertices is globally rigid in $\mathbb{R}$ w.h.p.
\end{conjecture}


\paragraph{Phase transition.}
The common theme of Theorems \ref{thm:main} and \ref{thm:sparse} is that while on general graphs injective global rigidity in $\mathbb{R}$ could be significantly different from generic global rigidity (i.e. $2$-connectivity), on random graphs they behave the same. According to Theorem~\ref{thm:LowerBoundGlobalRigidity} however the two exhibit different behaviour in terms of their phase transition. On the one hand we know  that the emergence of a linear sized 2-connected (hence generically globally rigid in $\mathbb{R}$) component happens in $G(n,p)$ around edge probability $\frac{1}{n}$, on the other hand our theorem shows that a linear-sized injectively globally rigid component appears only after  $\frac{1.1}{n}$. Since Theorem~\ref{thm:sparse} implies that some constant in place of $1.1$ would be sufficient to guarantee such phase transition, it is natural to wonder how large this must be. 

\begin{problem} What is the infimum of those constants $C$ for which there exists $\delta=\delta(C)>0$ such that in the random graph $G\sim G(n,C/n)$ w.h.p. there exists a subset $V' \subseteq [n]$ of size $|V'| \geq \delta n$ for which the induced subgraph $G[V']$ is globally rigid in $\mathbb{R}$. 
\end{problem}

\paragraph{Algorithmic problem.} We note that in the case of a fixed {injective} function $f$, Benjamini and Tzalik \cite{benjamini2022determining}, as well as Gir\~{a}o, Illingworth, Michel, Powierski, and Scott~\cite{girao2023reconstructing}, also considered the algorithmic problem of 
finding an injective function $f':V(G) \rightarrow \mathbb{R}$ with $|{f'(x)-f'(y)}|= | f(x)-f(y)|$ for every edge $xy\in E(G)$ when $G$ is a random graph. 
They obtain algorithms with polynomial expected running time. In our setup, where we generate only one random graph to reconstruct any {injective} function $f$, our proof does not provide any insight on how to find such $f'$. Note that there is a trivial algorithm with running time $O(|E(G)|2^n)$: take a BFS (or DFS) ordering of the vertices in $G$ and then embed each next vertex both possible ways (either $+$ or $-$ the edge length to its parent vertex), and at the end check whether all pairwise distances are correct. We wonder whether this could be improved. 


\begin{problem}
    Find an algorithm $\mathcal{A}$ with the following property: Let $G \sim G(n,p)$ for $p \gg \log n / n$. Then w.h.p.\ $G$ is such that, for any injective $f \colon V(G) \rightarrow \mathbb{R}$, $\mathcal{A}(G, f)$ finds in polynomial time (depending only on $n$) {an injective} function $f'$ satisfying $|f'(x)-f'(y)| = |f(x)-f(y)|$ for every edge $xy\in E(G)$.
\end{problem}

\paragraph{Higher dimensions.} Finally, while one cannot hope for an extension of Theorem \ref{thm:main} to $\mathbb{R}^d$ for $d \ge 2$, it is conceivable that a statement of Theorem \ref{thm:sparse} is true for any $d \ge 2$. Even showing this for a given $f \colon [n] \rightarrow \mathbb{R}$ is an open problem, already suggested in \cite{girao2023reconstructing}, with some recent progress by Barnes, Petr, Portier, Randall Shaw, and Sergeev \cite{barnes2024reconstructing}. Here we state the global rigidity version.

\begin{problem}
    Show that, for every integer $d\geq 2$ and $\eps > 0$, there exists $C > 0$ such that in the random graph $G \sim G(n,C/n)$  w.h.p. there exists a subset $V' \subseteq [n]$ of size $|V'| \ge (1 - \varepsilon)n$ for which the induced subgraph $G[V']$ is globally rigid in $\mathbb{R}^d$.
\end{problem}

\bibliographystyle{abbrv}
\bibliography{bibliography.bib}

\end{document}